\documentclass[12pt]{article}

\usepackage{amsmath, amssymb, amsthm, amsfonts, mathtools, array, xfrac, float, multirow, indentfirst, url, enumerate, comment, fullpage, afterpage, pifont, makecell, xcolor,mathrsfs,enumitem}

\usepackage[numbers]{natbib}

\providecommand{\noopsort}[1]{}
\usepackage{hyperref}
\usepackage{booktabs, siunitx, xcolor, graphicx, titlesec}

\allowdisplaybreaks

\newtheorem{theorem}{Theorem}
\newtheorem{conjecture}{Conjecture}
\newtheorem{proposition}{Proposition}

\newtheorem{lemma}{Lemma}

\theoremstyle{definition}

\let\svthefootnote\thefootnote
\newcommand\freefootnote[1]{%
  \let\thefootnote\relax%
  \footnotetext{#1}%
  \let\thefootnote\svthefootnote%
}

\definecolor{vio}{RGB}{118, 120, 238}

\newcommand{\mcA}{\mathcal{A}}

\newlength{\bibitemsep}
\newlength{\bibparskip}
\let\oldthebibliography\thebibliography
\renewcommand\thebibliography[1]{%
  \oldthebibliography{#1}%
  \setlength{\parskip}{\bibitemsep}%
  \setlength{\itemsep}{\bibparskip}%
}

\title{An update on the Linnik--Goldbach problem}

\author{Daniel R. Johnston\\
Max Planck Institute for Mathematics, Bonn, Germany \\ johnston@mpim-bonn.mpg.de \and 
Tim Trudgian \\
School of Science, UNSW Canberra, Australia \\
timothy.trudgian@unsw.edu.au}

\date{}

\begin{document}

\maketitle

\freefootnote{DRJ and TT are supported by Australian Research Council Discovery Project DP240100186.}
\freefootnote{Key phrases: Linnik--Goldbach problem, Romanov problem, Goldbach's conjecture, sieve methods}
\freefootnote{2020 MSC codes: 11N36, 11P32 (Primary) 11M26, 11P55 (Secondary)}

\noindent
\textit{RHB at 50 proved that seven 2's would do.\\ Twenty-five years later, we show six in this review.\\ We wish him to his century when proofs of 5 are due!}
\begin{abstract} 
\noindent 
We consider the Linnik--Goldbach problem of writing all large even integers as the sum of two primes and a fixed number of powers of 2. We show that, under the generalised Riemann hypothesis, one can use 6 powers of two. In addition, we discuss refinements to the unconditional case and to the related problem of Romanov in expressing a positive proportion of odd numbers as the sum of a prime and a power of 2.
\end{abstract}
\section{Introduction}
We refine the numerical bounds for the Linnik--Goldbach problem --- a well-studied problem in number theory. The main source of improvement arises from new bounds on Goldbach representations due to Lichtman~\cite{lichtman2023primes}. However, there are other minor optimisations we employ. The goal of this paper is to both showcase these new bounds, and also survey the most recent techniques for tackling the Linnik--Goldbach problem and the related problem of Romanov. Throughout, $p$, possibly with a subscript (e.g.~$p_1,p_2,p_3$), is assumed to be prime. Moreover, $\varphi(\cdot)$ denotes the Euler totient function and $(a,b)$ stands for $\gcd(a,b)$.

\subsection{The Linnik--Goldbach problem}
Linnik's approximation of Goldbach's conjecture asks for the smallest $K\geq 0$ such that every large even integer can be expressed as the sum of two primes and $K$ powers of 2:
\begin{equation}\label{Linnik}
    n=p_1+p_2+2^{\nu_1}+\cdots+2^{\nu_K}.
\end{equation}

Linnik \cite{linnik1953addition} was the first to show that $K$ exists; the first values of $K$ were very large (see e.g.\ \cite{liu1998number}). However, after a series of improvements, Heath-Brown and Schlage-Puchta \cite{heath2002integers} were able to obtain $K=7$ assuming the generalised Riemann hypothesis (GRH). Unconditionally, they were able to use some zero-density results by Heath-Brown \cite{heathcanada} to show\footnote{Included in \cite{heath2002integers} is a note about unpublished work by Elsholtz, which shows that $K=12$ is admissible.} that $K=13$ is admissible. Pintz and Ruzsa \cite{pintz2003linnik,pintz2020linnik} independently obtained the same result ($K=7$) under GRH and $K=8$ unconditionally. We show it is now possible to obtain $K=6$ under GRH. 

\begin{theorem}\label{GRHGoldLin}
    Assuming GRH, every sufficiently large even integer can be expressed as the sum of two primes and $K=6$ powers of $2$.  
\end{theorem}

The main estimate required to obtain Theorem~\ref{GRHGoldLin} is an upper bound on Goldbach representations due to Lichtman~\cite{lichtman2023primes}, which we detail in Sections~\ref{uppersubsect} and~\ref{uppersect}. Later in Table~\ref{PRtable}, we display the parameters required to achieve further improvements, both conditionally and unconditionally. Notably, we fall just short of obtaining $K=7$ unconditionally with our setup. However, recently announced work of Maynard, Pandey and Radziwi\l\l\, on exponential sums over primes should allow one to easily overcome this barrier and attain $K=7$. We discuss this further in Section~\ref{Maynardsect}. We also remark that $K=4$ is possible upon assuming the Elliott--Halberstam conjecture (see Conjecture~\ref{Ellhcon} and Theorem~\ref{EHthm}). 

\subsection{The Romanov problem}\label{romintro}

A very closely related problem to the Linnik--Goldbach problem is the \emph{Romanov problem}, which concerns $n$ that can be expressed as the sum of a prime and a power of 2:
\begin{equation*}\label{romeq}
    n=p+2^a.
\end{equation*}
In this context, we define
\begin{equation*}
    d(N):=\frac{|\{n\leq N:n=p+2^a\}|}{N}
\end{equation*}
to be the density of such numbers less than $N$, and
\begin{equation*}\label{limdefs}
    \underline{d}=\liminf_{N\to\infty}d(N)
\end{equation*}
be the lower density. The fact that $\underline{d}>0$ is due to Romanov~\cite{romanoff1934}, and an explicit value for the lower bound for $\underline{d}$ is often referred to as \emph{Romanov's constant}. Currently, the best known lower bound for $\underline{d}$ is
\begin{equation}\label{elsholtzlb}
    \underline{d}\geq 0.10788.    
\end{equation} 
due to\footnote{A slightly weaker bound is actually given in~\cite[Theorem~1]{elsholtz2018}. However, upon repeating the relevant computations we obtained the bound given in \eqref{elsholtzlb}, which is also stated at the end of page 721 of~\cite{elsholtz2018}.} Elsholtz and Schlage-Puchta~\cite{elsholtz2018}. Given the improvement we attain in the Goldbach-Linnik problem, it seems natural that an improvement should be possible in the Romanov problem too. However, due to a technicality in Elsholtz and Schlage-Puchta's argument, this is actually not possible. Instead, we must revert to an older technique for bounding $\underline{d}$ due to Pintz~\cite{pintz2006}. Here, Lichtman's new bound for Goldbach representations leads to $\underline{d}\geq 0.10695$. Although this improves upon Pintz's original bound, it unfortunately falls short of beating Elsholtz and Schlage-Puchta's bound~\eqref{elsholtzlb}. 

In any case, given the recent activity in improving bounds for Goldbach representations, it seems that sharper bounds in the Romanov problem are within reach. For this reason, in Section~\ref{Romsect}, we include additional discussion on the Romanov problem, and a table of parameters required for further refinements. 

\subsection{Upper bounds for Goldbach representations}\label{uppersubsect}
 
For an even number $N>2$, let \begin{equation}\label{Gndef}
    G(N)=\#\{(p_1,p_2)\in\mathbb{N}^2:2<p_1,p_2<N,\:p_1+p_2=N\}
\end{equation}
denote the number of Goldbach representations of $N$. Bounds for $G(N)$ have historically been linked to the Linnik--Goldbach and Romanov problems. We remark that Hardy and Littlewood~\cite[Conjecture~A]{hardy1923some} conjectured that 
\begin{equation*}\label{HLeq}
    G(N)\sim \frac{2C_0N}{(\log N)^2}\prod_{\substack{p\mid N\\p>2}}\frac{p-1}{p-2},
\end{equation*}
where
\begin{equation}\label{twinpconst}
    C_0=\prod_{p>2}\left(1-\frac{1}{(p-1)^2}\right)=0.66016\ldots
\end{equation}
is the twin-prime constant. Although a non-zero lower bound on $G(N)$ is out of reach, we have upper bounds of the form
\begin{equation}\label{Cstareq}
    G(N)\leq(C^*+\varepsilon)\frac{C_0N}{(\log N)^2}\prod_{p\mid N}\frac{p-1}{p-2},
\end{equation}
with $C^*\geq 2$ and $\varepsilon>0$ arbitrarily small. The value $C^*=8$ is commonly cited in the literature, and readily follows from classical sieve methods and the Bombieri--Vinogradov theorem (see~\cite[Theorem~2]{bombieri1966small} and \cite[Theorem~3.1]{heathsieves}). Using carefully weighted sieves, the value of $C^*$ can be lowered, with Chen~\cite{chen1978goldbach} proving $C^*=7.8342$ in 1978, and Wu~\cite{wu2004chen} proving $C^*=7.8209$ in 2004.

In recent years, there has been breakthrough work on variants of the Bombieri--Vinogradov theorem, allowing for smaller admissible values of $C^*$. The foundation was laid down by Maynard~\cite{maynard2025one,maynard2025two,maynard2025three} and enhanced by Lichtman~\cite{lichtman2023primes,lichtman2025modification} and Pascadi~\cite{pascadi2025exponents}. Notably, in~\cite{lichtman2023primes} the value ${C^*=6.7814}$ is attained, a significant improvement on previous work. 

\begin{proposition}[{\cite[Theorem~1.2]{lichtman2023primes}}]\label{lichtprop}
    The bound~\eqref{Cstareq} holds with $C^*=6.7814$.
\end{proposition}

We remark that in~\cite[p.~4]{pascadi2025exponents}, Pascadi suggests that with an extension of his techniques, one should  be able to reduce $C^*$ further than the value given in Proposition~\ref{lichtprop}.

For the purpose of proving Theorem~\ref{GRHGoldLin}, we actually need a slight variation of Proposition~\ref{lichtprop} (see Proposition~\ref{Cstarstarprop}). This is discussed in the next section.

\section{Statement of sieve upper bound results}\label{uppersect}

In this section, we detail Lichtman's recent upper bound for Goldbach representations, and also provide a variant of Proposition~\ref{lichtprop} for use in the Linnik--Goldbach problem. Here, it is useful for the reader to have a basic understanding of sieve-theoretic arguments, which are detailed in texts including~\cite{halberstam1974sieve,greaves2013sieves,friedlander2010opera}.

In the standard approach for the Linnik--Goldbach problem, we require upper bounds for
\begin{equation}\label{rnhdef}
    R(N,h):=\#\{p_1,p_2\leq N:p_1-p_2=h\}.
\end{equation}
In the following subsections, we first outline the standard argument to obtain bounds for $G(N)$, then describe how in~\cite{lichtman2023primes}, Lichtman managed to obtain Proposition~\ref{lichtprop}. Afterwards, we discuss the very minor modification to obtain a bound for $R(N,h)$ as opposed to $G(N)$.

\subsection{Bounding $G(N)$}
To bound $G(N)$, defined in~\eqref{Gndef}, one may equivalently bound the number of primes in 
\begin{equation}\label{AGdef}
    \mcA=\mcA(N):=\{N-p:p\leq N,(p,N)=1\}.
\end{equation}
Via a standard sifting argument (see~\cite[Theorem~3.11]{halberstam1974sieve}) an upper bound for $G(N)$ can be found by approximating the size of the sets
  $  \mcA_d:=\{a\in\mcA:d\mid a\},
$
where $d$ is square-free and coprime to $N$. In particular, one has
\begin{equation}\label{Adbound1}
    \#\mcA_d=\pi(N;d,N)+O(\log N),
\end{equation}
where
 $   \pi(N;d,a):=\#\{p\leq N:p\equiv a\ \text{(mod $d$)}\}
$
and the $O(\log N)$ term in~\eqref{Adbound1} comes from considering the number of primes $p$ with ${(p,N)>1}$. In turn, one is led to computing an averaged form of the error
\begin{equation*}
    E(N;d,a):=\pi(N;d,a)-\frac{\pi(N)}{\varphi(d)},
\end{equation*}
where $\pi(N):=\pi(N;1,1)$. The classical approach (as employed in the proof of \cite[Theorem~3.11]{halberstam1974sieve}) is to use the Bombieri--Vinogradov theorem, which states that for any $\varepsilon>0$,
\begin{equation}\label{bomvineq}
    \sum_{d\leq x^{\theta-\varepsilon}}\max_{(a,d)=1}\left|E(x;d,a)\right|\ll_{\varepsilon,A}\frac{x}{(\log x)^A},
\end{equation}
with $\theta=1/2$. Here, $\theta$ is referred to as the \emph{level of distribution} of primes: a central problem in analytic number theory is to prove that higher values of $\theta$ are admissible. Essentially, the higher the value of $\theta$, the better sieve-theoretic estimates are possible, with $\theta=1/2$ readily giving a constant of $C^*=8$ in~\eqref{Cstareq}. The Elliott--Halberstam conjecture asserts that $\theta=1$ is possible, which would lead to a constant of $C^*=4$ in~\eqref{Cstareq}.
\begin{conjecture}[Elliott--Halberstam]\label{Ellhcon}
    The bound~\eqref{bomvineq} holds with $\theta=1$.
\end{conjecture}
In practice, one does not require the full statement of the Bombieri--Vinogradov theorem. The absolute values in~\eqref{bomvineq} can be replaced by a suitable weight $\lambda(d)$. Indeed, in~\cite[Theorem~1.7]{lichtman2023primes}, Lichtman shows that
\begin{equation}\label{tripeq}
    \sup_{0<|a|<x^{1+\varepsilon}}\sum_{d\leq x^{\theta-\varepsilon}}\lambda(d)E(x;d,a)\ll_{\varepsilon,A}\frac{x}{(\log x)^A},
\end{equation}
where $\theta=153/256\approx0.597$ and $\lambda(d)$ is \emph{triply well-factorable of level $x^{\theta}$} as defined in~\cite[Definition~1.4]{lichtman2023primes}. Compared to previous work~\cite{maynard2025two,lichtman2025modification}, the main novelty of Lichtman's result is the uniformity in the residue $a$, which is required when bounding $G(N)$.

Unfortunately, the weights $\widetilde{\lambda}(d)$ arising\footnote{In particular, the weights introduced by Iwaniec in~\cite{iwaniec1980}.} in the linear sieve are not triply well-factorable but merely ``doubly" well-factorable. Because of this, Lichtman provides a technical result~\cite[Proposition~6.6]{lichtman2023primes} which shows that one can get a level of distribution with the sieve weights $\widetilde{\lambda}(d)$ that is \emph{at most} $\theta=153/256$, but often lower depending on the divisor structure of the moduli $d$ under consideration. In any case, Lichtman goes beyond the level of distribution offered by the Bombieri--Vinogradov theorem, and obtains $C^*=6.7814$ in~\eqref{Cstareq} using a procedure of Wu~\cite{wu2004chen} in~\cite[\S 7]{lichtman2023primes}.

\subsection{Bounding $R(N,h)$}\label{rnhsect}
Compared to $G(N)$ (cf.~\eqref{AGdef}), to bound $R(N,h)$ one considers the set
\begin{equation*}\label{A2def}
    \mcA'=\mcA'(N,h):=\{p+h:p\leq N,\:(p,h)=1\}
\end{equation*}
and via a sifting argument one is required to approximate
\begin{equation}\label{A2ddef}
    \#\mcA'_d:=\#\{a\in\mcA':d\mid a\}=\pi(N;d,-h)+O(\log h)
\end{equation}
with $d$ square-free and $(d,h)=1$. We note that the expressions~\eqref{Adbound1} and~\eqref{A2ddef} are identical but with the residue class $N$ replaced with $-h$. So, since the results in~\cite{lichtman2023primes} (see~\eqref{tripeq}) are uniform in the residue class, the argument required to bound $R(N,h)$ is identical to that of bounding $G(N)$. In particular, one obtains the following, which mirrors Proposition~\ref{lichtprop}.
\begin{proposition}\label{Cstarstarprop}
    Let $\varepsilon>0$. Then, for all $h<N$,
    \begin{equation}\label{gapsieveeq}
        R(N,h)\leq (C_1+\varepsilon)\frac{C_0N}{(\log N)^2}\prod_{\substack{p\mid h\\p>2}}\frac{p-1}{p-2}
    \end{equation}
    with $C_1=6.7814$ and $C_0$ as in \eqref{twinpconst}.
\end{proposition}
A key feature of Proposition~\ref{Cstarstarprop} is the uniformity of $h<N$. If instead $h<(\log N)^D$ for some $D>0$, then a ``small residue" sieving argument could be used, leading to the lower value of $C_1=6.458$ as per \cite[Theorem~1.1]{lichtman2023primes}.

To conclude this section, we also make note of the improvement one gets upon assuming the Elliott--Halberstam conjecture (Conjecture~\ref{Ellhcon}). In particular, for $R(N,h)$ the situation is exactly the same as the case of $G(N)$, whereby replacing the Bombieri--Vinogradov theorem (level of distribution $\theta=1/2$) with the Elliott--Halberstam conjecture (level of distribution $\theta=1$) gives $C_1=4$.

\begin{proposition}\label{EHprop}
   Assuming the Elliott--Halberstam conjecture, one may take $C_1=4$ in~\eqref{gapsieveeq}.
\end{proposition}

\section{The Linnik--Goldbach problem}\label{GLsect}
In this section, we explore the modern framework of the Linnik--Goldbach problem, both in the GRH and non-GRH cases, and prove Theorem~\ref{GRHGoldLin}. We include Table~\ref{PRtable} to show what further results are possible if $C_1$ in~\eqref{gapsieveeq} is lowered. Throughout, a basic understanding of the circle method is assumed, such as that found in the introductory text~\cite{murty2023}.

\subsection{The method of Pintz and Ruzsa}\label{PRsect}
We begin by detailing the core argument of Pintz and Ruzsa~\cite{pintz2003linnik,pintz2020linnik}, which is the most recent approach to the Linnik--Goldbach problem. The key result is Theorem~\ref{PRthm}, which allows one to obtain a valid value for $K$ in the Linnik--Goldbach problem, by inputting well-studied number-theoretic constants. The overall method is very similar to previous approaches, including that of Heath-Brown and Schlage-Puchta~\cite{heath2002integers} and a subsequent refinement by Platt and Trudgian~\cite{platt2015linnik}. However in~\cite{pintz2020linnik}, Pintz and Ruzsa make use of an ``explicit formula" of Pintz~\cite{pintz2023}, which allows for great flexibility when applying the circle method. 

In what follows, we set $n=N$ to be a sufficiently large even number to be studied in the context of the Linnik--Goldbach problem (see~\eqref{Linnik}) and $L = [\log N/\log 2]$.
In~\cite{pintz2020linnik}, a lower order term is also included in the definition of $L$ for technical purposes. However, we omit this term as it has no impact on the final asymptotic result.

In the setup for the Linnik--Goldbach problem, one studies two representation functions
\begin{equation}\label{rpdef}
    r_k'(n)=\sum_{n=p+2^{\nu_1}+\cdots+2^{\nu_k}}\log p
\end{equation}
and
\begin{equation*}
    r_k''(n)=\sum_{n=p_1+p_2+2^{\nu_1}+\cdots+2^{\nu_k}}\log p_1\log p_2,
\end{equation*}
where one wants to show $r_k''(N)>0$. To do so, one uses $L^1$ and $L^2$ estimates for $r_k'(n)$, combined with Cauchy--Schwarz. The $L^1$ estimate is the following asymptotic.

\begin{lemma}[{cf.~\cite[Lemma~14]{gallagher1975}}]\label{L1lem}
    For any fixed $k\geq 1$, one has $\sum_{n\leq N}r_k'(n)\sim NL^k$.
\end{lemma}
\begin{proof}
    Write 
    \begin{equation*}
        \sum_{n\leq N}r_k'(n)=\sum_{p\leq N}f(N-p)\log p
    \end{equation*}
    where $f(x)$ is the number of $k$-tuples $v_1,\ldots,v_k$ with $v_i\leq L$ and
    \begin{equation*}
        2^{v_1}+\cdots+2^{v_k}\leq x.
    \end{equation*}
    Notably,
    \begin{equation}\label{fxbounds}
        \left\lfloor\frac{\log (x/k)}{\log 2}\right\rfloor^k\leq f(x)\leq\left(\frac{\log x}{\log 2}\right)^k
    \end{equation}
    so that $f(x)\sim(\log x/\log 2)^k$. Now, let $\theta(x)=\sum_{p\leq x}\log p$ denote the Chebyshev weighted prime-counting function. Since $f$ is a non-negative increasing function,
    \begin{equation}\label{fNpupper}
        \sum_{p\leq N}f(N-p)\log p\leq f(N)\theta(N)
    \end{equation}
    and
    \begin{equation}\label{fNPlower}
        \sum_{p\leq N}f(N-p)\log p\geq  \sum_{p\leq N-N/L}f(N-p)\log p\geq\theta(N-N/L)f(N/L).
    \end{equation}
    The result then follows by applying~\eqref{fxbounds} and the prime number theorem to~\eqref{fNpupper} and~\eqref{fNPlower}.
\end{proof}
To obtain the $L^2$ bound, Pintz and Ruzsa use a technical application of the circle method. Before stating their bound, we define the number-theoretic constants involved.

To begin with, $C_0$ is the constant \eqref{twinpconst}, $C_1$ is as in Proposition~\ref{Cstarstarprop}, and $R_0$ is given by
\begin{equation*}
    R_0=\sum_{d=1}^\infty\frac{f(2d-1)}{\epsilon(2d-1)},
\end{equation*}
where $f(n)$ is the multiplicative function defined by
\begin{equation*}\label{fmultdef}
    f(p^e)=
    \begin{cases}
        0,&\text{$p=2$ or $e\geq 2$}\\
        (p-2)^{-1},&\text{otherwise},
    \end{cases}
\end{equation*}
and where $\epsilon(d)$ is the multiplicative order of $2$ mod $d$. Existing computations of $R_0$ (see \cite[p.~55]{platt2015linnik}) yield that
\begin{equation}\label{R0bounds}
    1.93642<R_0<1.93656.
\end{equation}
Pintz and Ruzsa then use a function $A(k)$ related to representations of integers as the sum and difference of powers of $2$. Explicitly, $A(k)$ is defined as
\begin{equation*}
    A(k):=\lim_{L\to\infty}\left(\frac{S(k,L)}{2L^{2k}}-1\right),
\end{equation*}
where
\begin{align*}
    S(k,L)&:=\sum_{\substack{m=-\infty\\ m\neq 0}}^\infty r_{k,k}(m)\sigma(m),\\
    r_{k,k}(m,L)&:=\#\{(a_1,\cdots, a_{2k})\in \{0,\ldots ,L\}^{2k}:m=2^{a_1}+\cdots + 2^{a_k}-2^{a_{k+1}}-\cdots -2^{a_{2k}}\}
\end{align*}
and
\begin{equation*}
     \sigma(m):=2C_0\prod_{\substack{p\mid m\\p>2}}\frac{p-1}{p-2}.
\end{equation*}

Khalfalah and Pintz proved several properties of $A(k)$ in \cite[Theorem~1]{khalfalah2006representation}. They showed that $A(k)$ is decreasing with $k$ and $\lim_{k\to\infty}A(k)=0$. Moreover, $A(k)>2^{-2k-1}$ for all $k\geq 1$, and, in addition, the following explicit bounds hold
(\cite[Theorem~2]{khalfalah2006representation})
\begin{align}
    0.27835<A(1)&<0.27926, \qquad 0.05458<A(2)<0.05549,\notag 
    \\
0.012697<A(3)&<0.013598, \qquad    0.003091<A(4)<0.003992. \notag
\end{align}

For their application of the circle method, Pintz and Ruzsa use the standard definitions of major and minor arcs. In particular, one lets $P$ and $Q$ be such that
\begin{equation*}
    2\leq P<Q=\frac{N}{P}
\end{equation*}
and the major $(\frak{M})$ and minor $(\frak{m})$ arcs be given by
\begin{align}
    \frak{M}&=\bigcup_{q\leq P}\bigcup_{\substack{a=1\\(a,q)=1}}^q\left[\frac{a}{q}-\frac{1}{qQ},\frac{a}{q}+\frac{1}{qQ}\right] \notag \\
    \frak{m}&=[1/Q,1+1/Q]\setminus\frak{M}
    \notag.
\end{align}
From here, one then estimates (via Parseval's identity)
\begin{equation*}
    \sum_{1\leq n\leq N}(r_k'(n))^2\leq\int_0^1|S(\alpha)G^k(\alpha)|^2\mathrm{d}\alpha=\int_{\frak{M}}|S(\alpha)G^k(\alpha)|^2\mathrm{d}\alpha+\int_{\frak{m}}|S(\alpha)G^k(\alpha)|^2\mathrm{d}\alpha,
\end{equation*}
where
\begin{equation*}\label{SGdef}
    S(\alpha)=\sum_{p\leq N} e(p\alpha)\log p\quad\text{and}\quad G(\alpha)=\sum_{1\leq \nu\leq L}e(2^\nu\alpha).
\end{equation*}
A key quantity is an acceptable choice of the ``cut-offs" $P$ and $Q$. Assuming GRH, Pintz and Ruzsa take $P=\sqrt{N}L^{-8}$ and $Q=\sqrt{N}$. This is sufficient to give an asymptotic for the major arcs~\cite[Lemma~1]{pintz2003linnik} and a suitable minor arc estimate. 
Unconditionally, the work of Pintz~\cite{pintz2023} allows one to take $P$ as large as $P=N^{4/9-\varepsilon}$ to give an asymptotic for the major arcs. In general, it is preferable to take $P$ as large as possible, so we will just set 
\begin{align*}\label{Pdef}
    P&=\sqrt{N} L^{-8}\qquad\text{(assuming GRH)},\\
    P&=N^{4/9-\varepsilon}\qquad\text{(unconditionally)}
\end{align*}
with $\varepsilon>0$ fixed. In~\cite{pintz2020linnik}, Pintz and Ruzsa set $P\in[N^{0.4},N^{0.41}]$ but there is no harm in setting $P$ to be larger here. Namely, for $\alpha\in\frak{m}$ one just needs the minor arc estimate (\cite[Lemma~2]{pintz2020linnik})
\begin{equation}\label{minorarcpw}
    S(\alpha)\ll\left(\frac{N}{\sqrt{P}}+N^{4/5}+\sqrt{NP}\right)L^4\ll L^4 N^{4/5}
\end{equation}
to hold, which is valid for any $P\in[N^{0.4},N^{0.6}]$.

Next, we define the constant $C_2'$, explicitly given as (see~\cite[Section~4]{pintz2020linnik})
\begin{align}
    C_2'&=\frac{1}{2}(C_1+\varepsilon-2)R_0C_0+\left(1-\frac{\log P}{\log N}\right)\frac{\log 2}{2}.\label{C2def1}
\end{align}
We note that the $\log P/\log N$ term in~\eqref{C2def1} is omitted in~\cite{pintz2020linnik}, but as remarked on~\cite[p.~578]{pintz2020linnik} it may be incorporated for additional optimisation.

Finally, we introduce the parameter $c_1$, given by (see~\cite[\S 7]{pintz2003linnik} and~\cite[\S 5]{pintz2020linnik})
\begin{align}
    c_1=0.7163436\qquad\text{(assuming GRH)},\notag\\
    c_1=0.7894009\qquad\text{(unconditionally)}.\label{lilc1bounds}
\end{align}
To define $c_1$ in a precise manner, suppose that, for some $0<\sigma<1$, we have an estimate
\begin{equation}\label{Salphaeq}
    S(\alpha)\ll N^{\sigma+\varepsilon}
\end{equation}
for $\alpha\in\frak{m}$ and some arbitrarily small $\varepsilon>0$. Then $c_1$ is any constant such that there exists a set $\mathcal{E}$ with $\mu(\mathcal{E})=N^{1-2\sigma}$ for which
\begin{equation}\label{Gc1bound}
    |G(x)|=\left|\sum_{j=0}^{L-1}e(2^j x)\right|<c_1L
\end{equation}
holds for all $x\in[0,1]\setminus\mathcal{E}$. Thus, unconditionally one computes $c_1$ using $\sigma=4/5$ by~\eqref{minorarcpw}. Then, under GRH, Pintz and Ruzsa use $\sigma=3/4$, which follows from an exponential sum estimate of Hardy and Littlewood~\cite[Lemma~9]{hardy1923some}. The actual computation to obtain the values of $c_1$ in~\eqref{lilc1bounds} is nontrivial, and is explained in detail in~\cite[\S 4-7]{pintz2003linnik}.

With all of the notation above, Pintz and Ruzsa obtain the following $L^2$ estimate.
\begin{lemma}[{See \cite[Lemma~13]{pintz2003linnik} or \cite[Lemma~9]{pintz2020linnik}}]\label{L2lem}
    With $A(k)$, $C_2'$ and $c_1$ as defined above, one has, for $\varepsilon>0$ and sufficiently large $N$,
    \begin{equation*}
        \sum_{n\leq N}(r_k'(n))^2\leq 2NL^{2k}(1+A(k)+C_2'c_1^{2k-2}+\varepsilon).
    \end{equation*}
\end{lemma}

Finally, from here, one then obtains the following result, which allows us to compute a value of $K$ in the Linnik--Goldbach problem.

\begin{theorem}[{Pintz--Ruzsa \cite{pintz2003linnik,pintz2020linnik}}]\label{PRthm}
    Let $K\geq 2$ and $i,j\geq 1$ be such that $K=i+j$ with $i=j$ or $i=j+1$ depending on whether $K$ is even or not. Then, using the notation above, if
    \begin{equation}
        \sqrt{A(i)+C_2'c_1^{2i-2}}\sqrt{A(j)+C_2'c_1^{2j-2}} <1,\label{PReq}
    \end{equation}
    then every sufficiently large even number is the sum of two primes and $K$ powers of $2$. 
\end{theorem}

\begin{proof}
    To begin with, we note that by Lemma~\ref{L1lem}, the average value of $r_k'(n)$ (with $n$ odd) is $2NL^k$. Therefore, the difference
    \begin{equation*}
        s_k(n):=r_k'(n)-2L^k
    \end{equation*}
    satisfies
    \begin{equation*}
        \sum_{\substack{n\leq N\\2\nmid n}}s_k(n)=o(NL^k).
    \end{equation*}
    The goal now is to establish, for every large even $N$, the positivity of
    \begin{align}\label{sieq1}
        r_K''(N)&=\sum_{\substack{m+n=N\\2\nmid m\\2\nmid n}}r_i'(m)r_j'(n)\notag\\
        &=\sum_{\substack{m+n=N\\2\nmid m\\2\nmid n}}s_i(m)s_j(n)+2L^j\sum_{\substack{m\leq N\\2\nmid m}}s_i(m)+2L^i\sum_{\substack{n\leq N\\2\nmid n}}s_j(n)+4L^K\sum_{\substack{m+n=N\\2\nmid m\\2\nmid n}}1\notag\\
        &=\sum_{\substack{m+n=N\\2\nmid m\\2\nmid n}}s_i(m)s_j(n)+2L^{K}N+o(NL^{K}).
    \end{align}
    To deal with the sum over $s_i(m)s_j(n)$, we apply Cauchy--Schwarz, giving
    \begin{align}\label{sieq2}
        \sum_{\substack{m+n=N\\2\nmid m\\2\nmid n}}s_i(m)s_j(n)&\leq\left(\sum_{\substack{m\leq N\\2\nmid m}}s_i(m)^2\right)^{1/2}\left(\sum_{\substack{n\leq N\\2\nmid n}}s_j(n)^2\right)^{1/2}.
    \end{align}
    Now, by applying both the $L^1$ bound (Lemma~\ref{L1lem}) and the $L^2$ bound (Lemma~\ref{L2lem})
    \begin{align}\label{sieq3}
        \sum_{\substack{m\leq N\\2\nmid m}}s_i(m)^2&=\sum_{\substack{m\leq N\\2\nmid m}}(r_i'(m)-2L^i)^2\notag\\
        &=\sum_{\substack{m\leq N\\2\nmid m}}r_i'(m)^2-4L^i\sum_{\substack{m\leq N\\2\nmid m}}r_i'(m)+4L^{2i}\sum_{\substack{m\leq N\\ 2\nmid m}}1\notag\\
        &\leq 2NL^{2i}\left(A(i)+C_2'c_1^{2i-2}+\varepsilon\right)
    \end{align}
    and analogously for the sum over $s_j(n)^2$. Substituting~\eqref{sieq3} into~\eqref{sieq2} and then~\eqref{sieq1} completes the proof.
\end{proof}

\subsection{Proof of Theorem~\ref{GRHGoldLin}}
From Theorem~\ref{PRthm} and Proposition~\ref{Cstarstarprop}, the proof of Theorem~\ref{GRHGoldLin} readily follows.

\begin{proof}[Proof of Theorem~\ref{GRHGoldLin}]
    Setting $K=6$ in Theorem~\ref{PRthm}, assuming GRH and setting $C_1=6.7814$ (from Proposition~\ref{Cstarstarprop}) gives that the left-hand side of~\eqref{PReq} is bounded above by
    \begin{equation*}
        A(3)+C_2'(0.7163436)^4\leq  0.865<1.\qedhere
    \end{equation*}
\end{proof}

Given the recent breakthrough work of Lichtman~\cite{lichtman2023primes} and Pascadi~\cite{pascadi2025exponents}, it seems likely that the value of $C_1=6.7814$ (from Proposition~\ref{Cstarstarprop}) may be lowered further. In Table~\ref{PRtable} we thereby list the value of $C_1$ required to get different values of $K$ in the Linnik--Goldbach problem. Notably, we are very close to obtaining $K=7$ unconditionally. In the next subsection, we detail how recently announced work should allow one to overcome this $K=7$ barrier.

We also recall that assuming the Elliott--Halberstam conjecture (Conjecture~\ref{Ellhcon}) one may take $C_1=4$ (Proposition~\ref{EHprop}), giving the following result.

\begin{theorem}\label{EHthm}
    Assume the Elliott--Halberstam conjecture. Then the Linnik--Goldbach problem holds with $K=4$.
\end{theorem}

\begin{table}
\caption{The value of $C_1$ required to solve the Linnik--Goldbach problem for $K$ powers of 2, as computed via Theorem~\ref{PRthm}.}
\vspace{1em}
\begin{tabular}{|c|c|c|||c|c|c|}
    \hline
    $K$ & Assuming GRH & $C_1$ required &  $K$ & Assuming GRH & $C_1$ required\\
    \hline
        7 & Yes & 9.958 & 7 & No & 6.737 \\
        6 & Yes & 7.589 & 6 & No & 5.672\\
        5 & Yes & 5.859 & 5 & No & 4.782\\
        4 & Yes & 4.608 & 4 & No & 4.069 \\
        3 & Yes & 3.613 & 3 & No & 3.398 \\
        2 & Yes & 2.856 & 2 & No & 2.826\\
        \hline
    \hline
\end{tabular}
\label{PRtable}
\end{table}

\subsection{Bounds on exponential sums over primes}\label{Maynardsect}

Recall Vinogradov's exponential sum bound over the minor arcs (equation~\eqref{minorarcpw}):
\begin{equation*}
    S(\alpha)\ll\left(\frac{N}{\sqrt{P}}+N^{4/5}+\sqrt{NP}\right)L^4\ll L^4 N^{4/5}.
\end{equation*}
Recently, Maynard, Pandey and Radziwi\l\l\footnote{See for example, Maynard's \href{https://www.crmath.ca/rapports/2026/PNT/pint_workshop_titles_abstracts.pdf}{abstract} at the 2026 Probability in Number Theory workshop in Montreal.} have announced an improvement to this classical bound, obtaining a power saving on the critical $N^{4/5}$ term. In particular, they are able to reduce this term to $N^{19/24+\varepsilon}$.

In our notational setup for the Linnik--Goldbach problem, this amounts to setting ${\sigma=19/24}$ in~\eqref{Salphaeq}. Consequently, one can use a lower value (unconditionally) for $c_1$ in Theorem~\ref{PRthm}. To test this numerically, we used the existing code by Languasco\footnote{Available at \href{https://codeocean.com/capsule/5525188/tree/v2}{codeocean.com/capsule/5525188/tree/v2}.}, first used in \cite{Ales}, to compute a valid value for $c_1$ in~\eqref{Gc1bound} with $\sigma=19/24$. This gives the value of
\begin{equation*}
    c_1=0.77779.
\end{equation*}
Substituting this new value of $c_1$ into Theorem~\ref{PRthm} with $K=7$ and $C_1=6.7814$ (from Proposition~\ref{Cstarstarprop}) yields
\begin{equation*}
    \sqrt{A(4)+C_2'c_1^{6}}\sqrt{A(3)+C_2'c_1^4}\leq 0.934,
\end{equation*}
thereby showing that every sufficiently large even number is the sum of two primes and $7$ powers of $2$. Given the preliminary nature of Maynard, Pandey and Radziwi\l\l's work, we have not listed this as a theorem. In any case, we even this small improvement to Vinogradov's bound (note $19/24=0.79166\ldots$) has a sizeable impact on the Linnik--Goldbach problem.

\section{Romanov's constant}\label{Romsect}
In this section we detail the impact of the constant $C_1$ (in Proposition~\ref{Cstarstarprop}) on the value one obtains for Romanov's constant as defined in Section~\ref{romintro}. We primarily discuss the classical approach of Pintz~\cite{pintz2006}, for which Lichtman's new Goldbach bounds are valid. We also outline the more recent method of Elsholtz and Schlage-Puchta~\cite{elsholtz2018}, which currently gives the best result of $\underline{d}\geq0.10788$.

\subsection{Pintz's method}
In~\cite{pintz2006}, Pintz gives a method for computing values of Romanov's constant. The method is very similar to the original approach of Romanov, with several optimisations added.

First, let
\begin{equation}\label{rndef}
    r(n):=\#\{(p,a):n=p+2^a\}
\end{equation}
denote the number of representations of $n$ as the sum of a prime and a power of two. Noting that $r(n)$ is an unweighted version of $r_k'(n)$ from the Linnik--Goldbach problem (see~\eqref{rpdef}), we attain $L^1$ and $L^2$ estimates for $r(n)$ which are very similar to those given for $r'(n)$ in Lemmas~\ref{L1lem} and~\ref{L2lem}. In particular, Pintz gives~\cite[Proposition~2]{pintz2006}
\begin{equation*}\label{S1bound}
    S_1(N):=\sum_{n\leq N}r(n)\sim\frac{N}{\log 2}
\end{equation*}
and for sufficiently large $N$~\cite[Lemma~3']{pintz2006}
\begin{equation*}\label{S2bound}
    S_2(N):=\sum_{n\leq N}r(n)^2\leq\widetilde{C}N,
\end{equation*}
where 
\begin{equation*}
    \widetilde{C}=\frac{1}{\log 2}\left(\frac{C_0(C_1+\varepsilon)R_0}{\log 2}+1\right)
\end{equation*}
with $C_0$ as in~\eqref{twinpconst}, $C_1$ as in Proposition~\ref{Cstarstarprop}, $R_0$ as in~\eqref{R0bounds}, and $\varepsilon>0$ arbitrary. From here, one can readily compute a lower bound for $\underline{d}$ via Cauchy--Schwarz. In particular, one has
\begin{equation*}
    S_1(N)^2\leq \left(\sum_{\substack{n\leq N\\r(n)>0}}1\right)S_2(N),
\end{equation*}
so that
\begin{equation*}
    \sum_{\substack{n\leq N\\r(n)>0}}1\geq\frac{N}{\widetilde{C}(\log 2)^2},
\end{equation*}
and thus $\underline{d}> 1/\widetilde{C}(\log 2)^2$. However, since $r(n)$ only takes integer values, Pintz is able to refine the use of Cauchy--Schwarz via the following lemma.

\begin{lemma}[{\cite[Lemma~4']{pintz2006}}]\label{pintzdenlem}
    Suppose that $b(n)\in\mathbb{N}\cup\{0\}$ for each $n\leq N$. Assume that
    \begin{equation*}
        \sum_{n=1}^Nb(n)=M\quad\text{and}\quad \sum_{n=1}^Nb(n)^2\leq DM.
    \end{equation*}
    Then,
    \begin{equation*}
        \#\{n\leq N:b(n)>0\}\geq\frac{\lceil D\rceil+\lfloor D\rfloor-D}{\lceil D\rceil\lfloor D\rfloor}M.
    \end{equation*}
\end{lemma}

So, setting $M=N/\log 2$ and $D=\widetilde{C}\log 2$ in Lemma~\ref{pintzdenlem}, along with Lichtman's value of ${C_1=6.7814}$, we obtain
\begin{equation*}
    \underline{d}\geq 0.10695,
\end{equation*}
which is just short of Elsholtz and Schlage-Puchta's bound $\underline{d}\geq 0.10788$. In Table~\ref{romtable}, we give lower bounds for $\underline{d}$ assuming different values of $C_1$. Notably, we see that if $C_1\leq 6.71$, then an improvement to Elsholtz and Schlage-Puchta's bound is attainable via this relatively simple method of Pintz. In addition, one attains $\underline{d}\geq 0.17277 $ under assumption of the Elliott--Halberstam conjecture (via Proposition~\ref{EHprop}).

\begin{table}
\caption{Lower bounds for $\underline{d}$ via Pintz's method for hypothetical values for $C_1$.}
\vspace{0.5em}
\begin{tabular}{|c|c c|}
    \hline
    $C_1$ & Lower bound for $\underline{d}$ & Notes\\
    \hline
        8 & 0.09163 & Classical value of $C_1$ from Bombieri--Vinogradov.\\
        7.8209 & 0.09362 & Implied by the work of Wu~\cite{wu2004chen}. \\
        6.7814 & 0.10695 & Lichtman's value of $C_1$. \\
        6.71& 0.10799 & Beats Elsholtz and Schlage-Puchta's record. \\
        4 & 0.17277 & Follows from the Elliott--Halberstam conjecture.\\
        2 & 0.31098 & Conjectural value for $C_1$. \\
    \hline
\end{tabular}
\label{romtable}
\end{table}

\subsection{Elsholtz and Schlage-Puchta's method}\label{ElSPsect}
In~\cite{elsholtz2018} Elsholtz and Schlage-Puchta use a novel idea to bound $\underline{d}$, which takes advantage of the fact that $r(n)$ (defined in~\eqref{rndef}) has an irregular distribution, depending on the residue class of $n$ modulo a fixed integer $\ell$. For example, modulo $3$, it is much more likely that $p+2^{\nu}\equiv 0$ mod $3$. This is because both primes and powers of two are uniformly distributed on the residue classes $1$ and $2$ (mod $3$). Extending on from this simple example, Elsholtz and Schlage-Puchta work with the much larger modulus $\ell=2^{24}-1$ for their computations.

With this setup in mind, Elsholtz and Schlage-Puchta rely on an upper bound for
\begin{equation*}\label{RNhkldef}
    R(N,h,k,\ell):=\#\{p_1,p_2\leq N:p_2\equiv k\ \text{(mod $\ell$)},\: p_1-p_2=h\}, 
\end{equation*}
which generalises $R(N,h)$ defined in~\eqref{rnhdef}. This allows them to obtain $L^1$ and $L^2$ estimates for the representation function
\begin{equation*}
    r(n,k,\ell):=\#\{(p,a):n=p+2^a\ \text{and}\ n\equiv k\thinspace(\ell) \},
\end{equation*}
in an analogous way to Pintz's moment estimates for $r(n)=r(n,1,1)$. Consequently, they compute the local densities 
\begin{equation*}\label{finaldinfeq}
    \underline{d}(k,\ell):=\liminf_{N\to\infty}\frac{\#\{n\leq N:r(n,k,\ell)>0\}}{N/\ell}.
\end{equation*}
via Cauchy--Schwarz, or more precisely, the refinement of Cauchy--Schwarz in Lemma~\ref{pintzdenlem}. Finally, a lower bound for $\underline{d}$ is computed via the identity

\begin{equation*}
    \underline{d}=\frac{\sum_{0\leq k<\ell}\underline{d}(k,\ell)}{\ell}.
\end{equation*}

Due to the highly skewed distribution of $\underline{d}(k,\ell)$ as $k$ varies, this method yields a modest improvement over Pintz's simpler approach.

In analogy with our refinement of the Goldbach--Linnik problem, the key question here is whether Elsholtz and Schlage-Puchta's bound on $R(N,h,k,\ell)$ can be improved with Lichtman's new sieve bounds. In~\cite{elsholtz2018}, the bound
\begin{equation}\label{romsieveeq}
    \sup_{\substack{k\in\mathbb{Z}\\(k,\ell)=1}}R(N,h,k,\ell)\leq (8+\varepsilon)\frac{C_0N}{\varphi(\ell)(\log N)^2}\prod_{\substack{p\mid \ell h\\p>2}}\frac{p-1}{p-2}
\end{equation}
is given for any fixed $\ell\geq 1$ and all $h<N$. The leading constant of $8$ here is obtained by the Bombieri--Vinogradov theorem, by a similar argument to the classical bound for $R(N,h)$ discussed in Section~\ref{uppersect}. Details of this classical argument can be found in~\cite[Theorem~3.12]{halberstam1974sieve}.

It seems possible then, that one could use Lichtman's theory in~\cite{lichtman2023primes} to reduce the constant $8$ to $6.7814$, in analogy with Proposition~\ref{Cstarstarprop}. However, this is unfortunately not possible. To see this, we note that to bound $R(N,h,k,\ell)$ one sifts the set
\begin{equation*}
    \mcA'(k,\ell)=\{p+h:p\leq N,\:(p,h\ell)=1,\:p\equiv k\ \text{(mod $\ell$)}\},
\end{equation*}
and so is led to estimate
\begin{equation*}\label{Akldeq}
    \#\mcA'(k,\ell)_d:=\#\{n\in\mcA'(k,\ell):d\mid n\}=\pi(N;d\ell,a)+O(\log \ell h),
\end{equation*}
where $a$ depends on $h,k,\ell$ and $d$ obtained via the Chinese remainder theorem from solving the congruences $p\equiv -h$ (mod $d$) and $p\equiv k$ (mod $\ell$). Note that by contrast, one had to study the fixed residue class $-h$ mod $d$ when bounding $R(N,h)$ in Section~\ref{rnhsect}.

Now, as discussed in Section~\ref{uppersect}, Lichtman's work then allows one to bound the sum
\begin{equation}\label{lichtv2}
    \sup_{0<|a|<x^{1+\varepsilon}}\sum_{d\leq x^{\theta-\varepsilon}}\lambda(d)E(x;d\ell,a),
\end{equation}
with $E(x;d\ell,a)=\pi(x;d\ell,a)-\pi(x)/\varphi(d\ell)$. However, since $a$ depends on $d$ in this setting bounding a sum of the form~\eqref{lichtv2} is not sufficient. In particular, one would require the supremum over $a$ to appear \emph{inside} the sum rather than \emph{outside} the sum as in~\eqref{lichtv2}.

We conclude by remarking that although Lichtman's theory is not able to improve~\eqref{romsieveeq}, standard sieve weighting procedures depending on the Bombieri--Vinogradov theorem can be applied to the problem of bounding $R(N,h,k,\ell)$. So, as also mentioned by Elsholtz and Schlage-Puchta \cite[p.~716]{elsholtz2018}, one could apply Wu's complicated sieve weighting procedure in~\cite{wu2004chen} to improve the $8$ in~\eqref{romsieveeq} to $7.8209$. This would give\footnote{Elsholtz and Schlage-Puchta refrain from listing this result as a theorem, due to the technical nature of Wu's argument, making it preferable for someone to carefully go through the sieve weighting argument to verify that it extends to this mod $\ell$ setting.} $\underline{d}\geq 0.11011$. However, if one could overcome the uniformity issue appearing in Lichtman's result and reduce the constant in~\eqref{romsieveeq} to $6.7814$, then a significantly better bound of $\underline{d}\geq 0.12532$ would follow from Elsholtz and Schlage-Puchta's method.

\section*{Acknowledgements}
We are grateful to Alex Pascadi and Roger Heath-Brown for enjoyable discussions on this topic, as well as the referee for some remarks that helped to improve our exposition. DRJ is grateful to the Max Planck Institute for Mathematics in Bonn
for its hospitality and financial support.


\begin{thebibliography}{}

\bibitem[Bombieri and Davenport, 1966]{bombieri1966small}
Bombieri, E. and Davenport, H. (1966).
\newblock Small differences between prime numbers.
\newblock {\em Proc. Roy. Soc. Ser. A}, 293(1432):1--18.

\bibitem[Chen, 1978]{chen1978goldbach}
Chen, J.~R. (1978).
\newblock On the {G}oldbach's problem and the sieve methods.
\newblock {\em Sci. Sinica}, 21(6):701--739.



\bibitem[Elsholtz and Schlage-Puchta, 2018]{elsholtz2018}
Elsholtz, C. and Schlage-Puchta, J.-C. (2018).
\newblock On {R}omanov's constant.
\newblock {\em Math. Z.}, 288(3-4):713--724.


\bibitem[Friedlander and Iwaniec, 2010]{friedlander2010opera}
Friedlander, J.~B. and Iwaniec, H. (2010).
\newblock {\em Opera de {C}ribro}.
\newblock American Mathematical Society, Providence RI.

\bibitem[Gallagher P. X., 1975]{gallagher1975}
Gallagher P. X. (1975).
\newblock Primes and powers of 2.
\newblock {\em Invent. Math.}, 2:125--142.

\bibitem[Greaves, 2013]{greaves2013sieves}
Greaves, G. (2013).
\newblock {\em Sieves in {N}umber {T}heory}.
\newblock Springer-Verlag, Berlin Heidelberg.

\bibitem[Halberstam and Richert, 1974]{halberstam1974sieve}
Halberstam, H. and Richert, H. (1974).
\newblock {\em Sieve {M}ethods}.
\newblock Academic Press, London.

\bibitem[Hardy and Littlewood, 1923]{hardy1923some}
Hardy, G.~H. and Littlewood, J.~E. (1923).
\newblock Some problems of ``partitio numerorum'', {III}: {O}n the expression
  of a number as a sum of primes.
\newblock {\em Acta Math.}, 44(1):1--70.


\bibitem[Heath-Brown, 1979]{heathcanada}
Heath-Brown, D.~R. (1979).
\newblock The density of zeros of {D}irichlet's ${L}$-functions.
\newblock {\em Canadian J. Math.}, 31(2):231--240.

\bibitem[Heath-Brown, 2002]{heathsieves}
Heath-Brown, D.~R. (2002).
\newblock Lectures on sieves.
\newblock {\em Available at
  \href{https://arxiv.org/abs/math/0209360}{arXiv:0209360}}.

\bibitem[Heath-Brown and Schlage-Puchta, 2002]{heath2002integers}
Heath-Brown, D.~R. and Puchta, J.-C. (2002).
\newblock Integers represented as a sum of primes and powers of two.
\newblock {\em Asian J. Math.}, 6(3):535--565.

\bibitem[Iwaniec, 1980]{iwaniec1980}
Iwaniec, H. (1980).
\newblock A new form of the error term in the linear sieve.
\newblock {\em Acta Arith.}, 37:307--320.

\bibitem[Khalfalah and Pintz, 2006]{khalfalah2006representation}
Khalfalah, A. and Pintz, J. (2006).
\newblock On the representation of {G}oldbach numbers by a bounded number of
  powers of two.
\newblock In {\em Elementare und analytische Zahlentheorie}, pages 129--142.
  Schr. Wiss. Ges. Johann Wolfgang Goethe Univ. Frankfurt am Main.
  
  \bibitem{Ales}
Languasco, A. and Zaccagnini, A. (2010).
\newblock On a Diophantine problem with two primes and $s$ powers of two.
\newblock{\em Acta Arith.}, 145:193--208.

\bibitem[Lichtman, 2023]{lichtman2023primes}
Lichtman, J.~D. (2023).
\newblock Primes in arithmetic progressions to large moduli, and {G}oldbach
  beyond the square-root barrier.
\newblock {\em Preprint available at
  \href{https://arxiv.org/abs/2309.08522}{arXiv:2309.08522}}.

\bibitem[Lichtman, 2025]{lichtman2025modification}
Lichtman, J.~D. (2025).
\newblock A modification of the linear sieve, and the count of twin primes.
\newblock {\em Algebra Number Theory}, 19(1):1--38.

\bibitem[Linnik, 1953]{linnik1953addition}
Linnik, Y.~V. (1953).
\newblock Addition of prime numbers with powers of one and the same number (in
  {R}ussian).
\newblock {\em Mat. Sbornik N.S.}, 74(32):3--60.

\bibitem[Liu et~al., 1998]{liu1998number}
Liu, J., Liu, M., and Wang, T. (1998).
\newblock The number of powers of 2 in a representation of large even integers
  {II}.
\newblock {\em Sci. China Ser. A}, 41(12):1255--1271.

\bibitem[Maynard, 2025a]{maynard2025one}
Maynard, J. (2025a).
\newblock Primes in {A}rithmetic {P}rogressions to {L}arge {M}oduli {I}:
  {F}ixed {R}esidue {C}lasses.
\newblock {\em Mem. Amer. Math. Soc.}, 306(1542).

\bibitem[Maynard, 2025b]{maynard2025two}
Maynard, J. (2025b).
\newblock Primes in {A}rithmetic {P}rogressions to {L}arge {M}oduli {II}:
  {W}ell-{F}actorable {E}stimates.
\newblock {\em Mem. Amer. Math. Soc.}, 306(1543).

\bibitem[Maynard, 2025c]{maynard2025three}
Maynard, J. (2025c).
\newblock Primes in {A}rithmetic {P}rogressions to {L}arge {M}oduli {III}:
  {U}niform {R}esidue {C}lasses.
\newblock {\em Mem. Amer. Math. Soc.}, 306(1544).

\bibitem[Murty and Sinha, 2023]{murty2023}
Murty, M. R. and Sinha, K. (2023).
\newblock {\em An {I}ntroduction to the {C}ircle {M}ethod}.
\newblock American Mathematical Society, Providence RI.

\bibitem[Pascadi, 2025]{pascadi2025exponents}
Pascadi, A. (2025).
\newblock On the exponents of distribution of primes and smooth numbers.
\newblock {\em Preprint available at
  \href{https://arxiv.org/abs/2505.00653}{arXiv:2505.00653}}.

\bibitem[Pintz, 2006]{pintz2006}
Pintz, J. (2006).
\newblock A note on {R}omanov's constant.
\newblock {\em Acta Math. Hungar.}, 112(1-2):1--14.

\bibitem[Pintz, 2023]{pintz2023}
Pintz, J. (2023).
\newblock A new explicit formula in the additive theory of primes with
  applications {I}. {T}he explicit formula for the {G}oldbach problem and the
  {G}eneralized {T}win {P}rime {P}roblem.
\newblock {\em Acta Arith.}, 210:53--94.

\bibitem[Pintz and Ruzsa, 2003]{pintz2003linnik}
Pintz, J. and Ruzsa, I.~Z. (2003).
\newblock On {L}innik's approximation to {G}oldbach's problem {I}.
\newblock {\em Acta Arith.}, 109(2):169--194.

\bibitem[Pintz and Ruzsa, 2020]{pintz2020linnik}
Pintz, J. and Ruzsa, I.~Z. (2020).
\newblock On {L}innik's approximation to {G}oldbach's problem {II}.
\newblock {\em Acta Math. Hungar.}, 161(2):569--582.

\bibitem[Platt and Trudgian, 2015]{platt2015linnik}
Platt, D.~J. and Trudgian, T.~S. (2015).
\newblock Linnik's approximation to {G}oldbach's conjecture, and other
  problems.
\newblock {\em J. Number Theory}, 153:54--62.


\bibitem[Romanoff, 1934]{romanoff1934}
Romanoff, N.~P. (1934).
\newblock \"{U}ber einige {S}\"atze der additiven {Z}ahlentheorie.
\newblock {\em Math. Ann.}, 109(1):668--678.


\bibitem[Wu, 2004]{wu2004chen}
Wu, J. (2004).
\newblock Chen's double sieve, {G}oldbach's conjecture and the twin prime
  problem.
\newblock {\em Acta Arith.}, 114(3):215--273.

\end{thebibliography}
\end{document}